\documentclass[11pt,a4paper]{article}
\usepackage[utf8]{inputenc}
\usepackage[english]{babel}
\usepackage[T1]{fontenc}
\usepackage{amsmath,amsfonts,amssymb,amsthm}
\usepackage{mathtools}
\usepackage{lmodern}
\usepackage{graphicx}
\usepackage{tikz}
\usepackage[hidelinks]{hyperref}
\usepackage[left=2cm,right=2cm,top=2cm,bottom=2cm]{geometry}
\author{Guillaume Olikier\footnotemark[2] \and Petar Mlinari\'{c}\footnotemark[3] \and P.-A. Absil\footnotemark[4] \and André Uschmajew\footnotemark[5]}
\title{The tangent cone to the real determinantal variety: various expressions and a proof\footnotemark[1]}

\newcommand{\N}{\mathbb{N}}
\newcommand{\R}{\mathbb{R}}
\newcommand{\C}{\mathbb{C}}

\newcommand{\cU}{\mathcal{U}}
\newcommand{\cV}{\mathcal{V}}

\newcommand{\ball}{B}
\DeclareMathOperator{\im}{im}
\newcommand{\ip}[2]{\langle #1 , #2 \rangle}
\newcommand{\norm}[1]{\lVert #1 \rVert}
\DeclareMathOperator*{\argmin}{argmin}

\newcommand{\proj}[2]{P_{#1}(#2)}

\newcommand{\gencone}[4]{{#1}_{#2}^{#4}(#3)} 
\newcommand{\tancone}[2]{\gencone{T}{#1}{#2}{}} 
\newcommand{\norcone}[2]{\gencone{N}{#1}{#2}{}} 

\newcommand{\oshort}[1]{\mkern 0.9mu\overline{\mkern-0.9mu#1\mkern-0.9mu}\mkern 0.9mu}
\newcommand{\ushort}[1]{\mkern 0.9mu\underline{\mkern-0.9mu#1\mkern-0.9mu}\mkern 0.9mu}

\newcommand{\tp}{\top}
\newcommand{\st}{\mathrm{St}}
\newcommand{\grass}{\mathrm{Gr}}
\DeclareMathOperator{\tr}{tr}
\DeclareMathOperator{\rank}{rank}

\newtheorem{theorem}{Theorem}[section]
\newtheorem{proposition}[theorem]{Proposition}
\newtheorem{lemma}[theorem]{Lemma}

\begin{document}
\renewcommand{\thefootnote}{\fnsymbol{footnote}}
\footnotetext[1]{The work of G. O. was supported by ERC grant 786854 G-Statistics from the European Research Council under the European Union's Horizon 2020 research and innovation program and by the French government through the 3IA Côte d'Azur Investments ANR-23-IACL-0001 managed by the National Research Agency. The work of P.-A. A. was supported by the Fonds de la Recherche Scientifique -- FNRS under Grant no T.0001.23. The work of A. U. was supported by the Deutsche Forschungsgemeinschaft (DFG, German Research Foundation) -- Projektnummer 506561557.}
\footnotetext[2]{Institute of Mathematics, \'{E}cole polytechnique fédérale de Lausanne (EPFL), 1015 Lausanne, Switzerland \mbox{(\href{mailto:guillaume.olikier@epfl.ch}{\nolinkurl{guillaume.olikier@epfl.ch}})}.}
\footnotetext[3]{Department of Mathematics, Faculty of Science, University of Zagreb, Bijeni\v{c}ka cesta 30, 10000 Zagreb, Croatia (\href{mailto:petar.mlinaric@math.hr}{\nolinkurl{petar.mlinaric@math.hr}}).}
\footnotetext[4]{ICTEAM Institute, UCLouvain, Avenue Georges Lemaître 4, 1348 Louvain-la-Neuve, Belgium \mbox{(\href{mailto:pa.absil@uclouvain.be}{\nolinkurl{pa.absil@uclouvain.be}})}.}
\footnotetext[5]{Institute of Mathematics \& Centre for Advanced Analytics and Predictive Sciences, University of Augsburg, 86159 Augsburg, Germany \mbox{(\href{mailto:andre.uschmajew@uni-a.de}{\nolinkurl{andre.uschmajew@uni-a.de}})}.}
\renewcommand{\thefootnote}{\arabic{footnote}}

\maketitle

\begin{abstract}
The set of real matrices of upper-bounded rank is a real algebraic variety called the real generic determinantal variety. An explicit description of the tangent cone to that variety is given in Theorem~3.2 of Schneider and Uschmajew [\textit{SIAM J. Optim.}, 25 (2015), pp. 622--646]. The present paper shows that the proof therein is incomplete and provides a proof. It also reviews equivalent descriptions of the tangent cone to that variety. Moreover, it shows that the tangent cone and the algebraic tangent cone to that variety coincide, which is not true for all real algebraic varieties.
\medskip

\medskip

\end{abstract}

\section{Introduction}
\label{sec:Introduction}
Let $\N$ be the set of nonnegative integers, $F$ a field, and $m, n, r \in \N$ such that $r < \min\{m, n\}$. The set
\begin{equation*}
F_{\le r}^{m \times n} \coloneq \left\{X \in F^{m \times n} \mid \rank X \le r\right\}
\end{equation*}
is the set of all matrices in $F^{m \times n}$ whose order-$(r+1)$ determinants are zero.
Thus, since the determinant is a polynomial function, $F_{\le r}^{m \times n}$ is an affine algebraic subset of $F^{m \times n}$, hence an affine algebraic variety, known as the \emph{generic determinantal variety} over $F$---term found in \cite{ArbarelloCornalbaGriffithsHarris,Harris} and adopted in this paper---or simply the \emph{determinantal variety} over $F$ \cite{LakshmibaiBrown}; see section~\ref{subsec:ElementsAlgebraicAnalyticGeometry} for a brief review of definitions and results from algebraic and analytic geometry.

Many machine learning or signal processing tasks such as dimensionality reduction, collaborative filtering, and signal recovery can be formulated as an optimization problem on $\R_{\le r}^{m \times n}$; see, e.g., references in~\cite[\S 1]{OlikierGallivanAbsil2026}. This has motivated a thorough study of its geometry~\cite{SchneiderUschmajew2015,HosseiniUschmajew2019,HosseiniLukeUschmajew2019,OlikierAbsil2022}.

In his 1908 book \emph{Formulario Mathematico}, Peano introduced an affine tangent cone to a subset of a Euclidean affine space and used it to formulate a necessary condition for optimality. This tangent cone was rediscovered independently by Bouligand and Severi in 1930. Historical investigations can be found in~\cite{DoleckiGreco2007,DoleckiGreco2011}. Since then, with the development of variational analysis, the tangent cone has proven to be a basic tool in constrained optimization to describe admissible search directions and to formulate optimality conditions. It is often called the Bouligand tangent cone, the contingent cone, or simply the tangent cone~\cite{AubinFrankowska,RockafellarWets,Mordukhovich1,Mordukhovich2,Mordukhovich2018}.

The tangent cone also appears in other fields of mathematics. In geometric measure theory, it plays an important role in the study of sets with positive reach~\cite{Federer,Federer1969,RatajZahle}. In algebraic and analytic geometry, the tangent cone to a real or complex algebraic or analytic variety enables to study the local geometry of the variety around its singular points; see references in section~\ref{subsec:ElementsAlgebraicAnalyticGeometry}. Over $\C$, the tangent cone always coincides with the \emph{algebraic} tangent cone. Over~$\R$, the tangent cone is always included in the algebraic tangent cone, and the inclusion can be strict.

Determinantal varieties over algebraically closed fields have been well studied in algebraic geometry; see, e.g.,~\cite[Chapter~II]{ArbarelloCornalbaGriffithsHarris},~\cite[\S 1C]{BrunsVetter},~\cite[Lecture~9]{Harris},~\cite[Chapter~10]{LakshmibaiBrown}, and~\cite{Tsakiris2024} and references therein. In particular, the algebraic tangent cone to the generic determinantal variety over an arbitrary algebraically closed field has been explicitly known since the eighties; see, e.g.,~\cite[(2.2)]{ArbarelloCornalbaGriffithsHarris} and~\cite[Example~20.5]{Harris}. In contrast, the algebraic tangent cone to the generic determinantal variety over $\R$ seems to be unavailable in the literature. Furthermore, for that variety, it is a priori unclear whether the tangent cone can be deduced from the algebraic tangent cone.

An explicit description of the tangent cone to the generic determinantal variety over $\R$ is given in~\cite[Theorem~3.2]{SchneiderUschmajew2015}; it is reviewed in section~\ref{sec:TangentConeRealGenericDeterminantalVariety} together with equivalent expressions. These expressions coincide with that of the algebraic tangent cone to the generic determinantal variety over algebraically closed fields, as detailed in sections~\ref{subsec:AlgebraicTangentConeRealGenericDeterminantalVariety} and~\ref{subsec:ExpressionViaGrassmann}. However, we have detected a gap in~\cite[proof of Theorem~3.2]{SchneiderUschmajew2015}, as detailed in section~\ref{sec:GapProofSU15Thm3.2}, and we are currently unaware of a complete proof in the literature.

A strategy to obtain a complete proof would be to adapt~\cite[proof of Theorem~6.1]{CasonAbsilVanDooren2013}. Indeed,~\cite[Theorem~6.1]{CasonAbsilVanDooren2013} describes explicitly the tangent cone to the intersection of $\R_{\le r}^{m \times n}$ and the unit sphere of $\R^{m \times n}$ endowed with the Frobenius inner product, which is also an affine algebraic subset of $\R^{m \times n}$. Since the intersection of two affine algebraic sets is an affine algebraic set, the tangent cone to the intersection of $\R_{\le r}^{m \times n}$ and the unit sphere admits the representation in~\eqref{eq:TangentVectorsToAlgebraicSetsAsTangentsToAnalyticArcsInTheSet}, on which~\cite[proof of Theorem~6.1]{CasonAbsilVanDooren2013} is based. Discarding the unit sphere constraint in that proof could yield a proof of~\cite[Theorem~3.2]{SchneiderUschmajew2015}.

In this paper, we bring four contributions to the study of the geometry of the generic determinantal variety over $\R$. First, in section~\ref{sec:GapProofSU15Thm3.2}, we describe the aforementioned gap in \cite[proof of Theorem~3.2]{SchneiderUschmajew2015}. Second, in section~\ref{sec:Proof}, we prove \cite[Theorem~3.2]{SchneiderUschmajew2015}. In contrast with the above-mentioned strategy based on~\cite[proof of Theorem~6.1]{CasonAbsilVanDooren2013}, our proof, rooted in the standard definition of the tangent cone in variational analysis, is shorter, explicit, and does not rely on algebraic geometry concepts. Third, in section~\ref{subsec:AlgebraicTangentConeRealGenericDeterminantalVariety}, we deduce that the algebraic tangent cone and the tangent cone to $\R_{\le r}^{m \times n}$ coincide. Fourth, in section~\ref{sec:DiscussionsComplements}, we provide an alternative proof of the part of \cite[Theorem~3.2]{SchneiderUschmajew2015} for which \cite[proof of Theorem~3.2]{SchneiderUschmajew2015} contains a gap, in two versions, geometric and algebraic, which follows more closely the strategy of \cite[proof of Theorem~3.2]{SchneiderUschmajew2015} than the proof from section~\ref{sec:Proof}. In view of its geometric formulation, this proof is less matrix-specific than the one from section~\ref{sec:Proof} and it may lend itself more nicely to extensions to other algebraic varieties.

The rest of the paper is organized as follows. Section~\ref{sec:ProjectionTangentCone} reviews the concepts of projection and tangent cone in a Euclidean vector space. Section~\ref{sec:MatrixNotation} introduces the matrix notation that is used throughout the paper. Section~\ref{sec:TangentConeRealGenericDeterminantalVariety} reviews several known descriptions of the tangent cone to $\R_{\le r}^{m \times n}$. A complete proof of these descriptions is proposed in section~\ref{sec:Proof}. The gap in~\cite[proof of Theorem~3.2]{SchneiderUschmajew2015} is described in section~\ref{sec:GapProofSU15Thm3.2}. Section~\ref{sec:AlgebraicTangentConeGenericDeterminantalVariety} focuses on the algebraic tangent cone to the generic determinantal variety. Conclusions are drawn in section~\ref{sec:Conclusion}. Section~\ref{sec:DiscussionsComplements} proposes discussions and complements.

\section{Projection and tangent cone in a Euclidean vector space}
\label{sec:ProjectionTangentCone}
Let $\mathcal{E}$ be a Euclidean vector space and $S$ a nonempty subset of $\mathcal{E}$.
The \emph{projection} of $x \in \mathcal{E}$ onto $S$ is $\proj{S}{x} \coloneq \argmin_{y \in S} \norm{x-y}$, which is a nonempty compact set if $S$ is closed~\cite[Example~1.20]{RockafellarWets}. When $\proj{S}{x}$ is a singleton, it is identified with its unique element. If $S$ is a linear subspace, then $P_S$ is a linear map from $\mathcal{E}$ to $S$.
The set $S$ is called a \emph{cone} if $x \in S$ implies $\alpha x \in S$ for all $\alpha \in [0, \infty)$~\cite[\S 3B]{RockafellarWets}.
A vector $v \in \mathcal{E}$ is said to be \emph{tangent} to $S$ at $x \in S$ if there exist sequences $(x_i)_{i \in \N}$ in $S$ converging to $x$ and $(t_i)_{i \in \N}$ in $(0, \infty)$ such that the sequence $(\frac{x_i-x}{t_i})_{i \in \N}$ converges to~$v$~\cite[Definition~6.1]{RockafellarWets}. The set of all tangent vectors to $S$ at $x \in S$ is a closed cone~\cite[Proposition~6.2]{RockafellarWets} called the \emph{tangent cone} to $S$ at $x$ and denoted by $\tancone{S}{x}$.
If $S$ is a submanifold of $\mathcal{E}$ locally around $x \in S$, then $\tancone{S}{x}$ coincides with the tangent space to $S$ at $x$~\cite[Example~6.8]{RockafellarWets}, whose orthogonal complement is the normal space to $S$ at $x$ and denoted by $\norcone{S}{x}$.
The following basic fact is used in section~\ref{sec:Proof}.

\begin{proposition}
\label{prop:TanConeToClosedConeAtTheOrigin}
If $S \subseteq \mathcal{E}$ is a closed cone, then $\tancone{S}{0} = S$.
\end{proposition}

\begin{proof}
Let us prove the inclusion $S \subseteq \tancone{S}{0}$. Let $v \in S$ and $(t_i)_{i \in \N}$ be a sequence in $(0, \infty)$ converging to $0$. Define $x_i \coloneq t_i v$ for all $i \in \N$. Then, $(x_i)_{i \in \N}$ is in $S$ and converges to $0$. Moreover, for all $i \in \N$, $\frac{x_i}{t_i} = v$. Thus, $(\frac{x_i}{t_i})_{i \in \N}$ converges to $v$ and hence $v \in \tancone{S}{0}$.

Let us prove the converse inclusion. Let $v \in \tancone{S}{0}$. By definition, there exist sequences $(x_i)_{i \in \N}$ in $S$ converging to $0$ and $(t_i)_{i \in \N}$ in $(0, \infty)$ such that $(\frac{x_i}{t_i})_{i \in \N}$ converges to $v$. Since $S$ is a cone, $(\frac{x_i}{t_i})_{i \in \N}$ is in $S$. Since $S$ is closed, the limit $v$ of $(\frac{x_i}{t_i})_{i \in \N}$ is in $S$.
\end{proof}

\section{Matrix notation}
\label{sec:MatrixNotation}
Let $F$ be a field. The \emph{image} (also called the \emph{range} or \emph{column space}) of $X \in F^{m \times n}$ is $\im X \coloneq X F^n$, which is a linear subspace of $F^m$ whose dimension is called the \emph{rank} of $X$ and denoted by $\rank X$. For every $\ushort{r} \in \{0, \dots, \min\{m, n\}\}$,
\begin{equation*}
F_{\ushort{r}}^{m \times n} \coloneq \left\{X \in F^{m \times n} \mid \rank X = \ushort{r}\right\}.
\end{equation*}
For every $p, q \in \N$, $0_{p \times q}$ is the zero matrix in $F^{p \times q}$ and $I_p$ is the identity matrix in $F^{p \times p}$. 

Let $F$ be $\R$ or $\C$. The \emph{conjugate transpose} of $X \in F^{m \times n}$ is $X^* \coloneq \oshort{X}^\tp$ \cite[\S 0.2.5]{HornJohnson}. For every $p, q \in \N$ such that $q \ge p$,
\begin{equation*}
\st_F(p, q) \coloneq \left\{U \in F^{q \times p} \mid U^* U = I_p\right\}
\end{equation*}
is a \emph{Stiefel manifold}; see, e.g.,~\cite[\S 3.3.2]{AbsilMahonySepulchre} for $F = \R$. The tensor product of two linear subspaces $\mathcal{U} \subseteq F^m$ and $\mathcal{V} \subseteq F^n$ can be regarded as a linear subspace of $F^{m \times n}$ as follows:
\begin{align}
\label{eq:TensorProductMatrixFactorization}
\mathcal{U} \otimes \mathcal{V} = \left\{Z \in F^{m \times n} \mid \im Z \subseteq \mathcal{U} \text{ and } \im Z^* \subseteq \mathcal{V}\right\};
\end{align}
see, e.g.,~\cite[\S 1.1.1]{Hackbusch} for $F = \R$. In what follows, $F^{m \times n}$ is endowed with the Frobenius inner product $\ip{X}{Y} \coloneq \tr Y^* X$~\cite[(5.2.7)]{HornJohnson}.

\section{The tangent cone to the real generic determinantal variety}
\label{sec:TangentConeRealGenericDeterminantalVariety}
Several expressions for the tangent cone to $\R_{\le r}^{m \times n}$ are reviewed in Theorem~\ref{thm:TanConeRealGenericDeterminantalVariety}, based on the stratification of $\R_{\le r}^{m \times n}$ by the rank, and in Theorem~\ref{thm:TanConeRealGenericDeterminantalVarietyMatrix}, through matrix equations.

The tangent cone to $\R_{\le r}^{m \times n}$ is closely related to the stratification of $\R_{\le r}^{m \times n}$ that is induced by the subadditive and lower semicontinuous function $\rank : \R^{m \times n} \to \{0, \dots, \min\{m, n\}\}$; see \cite{Hiriart-UrrutyLe} for more variational properties of the rank function. Specifically, $\R_{\le r}^{m \times n}$ can be partitioned as
\begin{equation*}
\R_{\le r}^{m \times n} = \bigcup_{\ushort{r}=0}^r \R_{\ushort{r}}^{m \times n},
\end{equation*}
and, for every $\ushort{r} \in \{0, \dots, \min\{m, n\}\}$, $\R_{\ushort{r}}^{m \times n}$ is an embedded submanifold of $\R^{m \times n}$ of dimension $(m+n-\ushort{r})\ushort{r}$~\cite[Example~8.14]{Lee2003}, whose closure is $\R_{\le \ushort{r}}^{m \times n}$ (see, e.g.,~\cite[Proposition~2.1]{OlikierAbsil2022}), and that is connected except if $m = n = \ushort{r}$ (see~\cite[Chapter~5, Proposition~1.14]{HelmkeMoore} or~\cite[Proposition~4.1]{HelmkeShayman1995}). The closure property implies that the stratification satisfies the condition of the frontier~\cite{Mather}.

The main purpose of this paper is to provide a short yet detailed and complete proof of Theorem~\ref{thm:TanConeRealGenericDeterminantalVariety}, which describes the tangent cone to the generic determinantal variety over $\R$. Such a proof can be found in section~\ref{sec:Proof}. The description is based on the tangent and normal spaces to $\R_{\ushort{r}}^{m \times n}$ and the projection onto the normal space, whose expressions are reviewed below.

\begin{theorem}
\label{thm:TanConeRealGenericDeterminantalVariety}
Let $X \in \R_{\le r}^{m \times n}$ and $\ushort{r} \coloneq \rank X$. Then,
\begin{align}
\tancone{\R_{\le r}^{m \times n}}{X}
\label{eq:TanConeRealGenericDeterminantalVarietyOrthogonal}
&= \tancone{\R_{\ushort{r}}^{m \times n}}{X} + \left(\norcone{\R_{\ushort{r}}^{m \times n}}{X} \cap \R_{\le r-\ushort{r}}^{m \times n}\right)\\
\label{eq:TanConeRealGenericDeterminantalVariety}
&= \tancone{\R_{\ushort{r}}^{m \times n}}{X} + \R_{\le r-\ushort{r}}^{m \times n}\\
\label{eq:TanConeRealGenericDeterminantalVarietyNormal}
&= \left\{Z \in \R^{m \times n} \mid \rank \proj{\norcone{\R_{\ushort{r}}^{m \times n}}{X}}{Z} \le r-\ushort{r}\right\}.
\end{align}
\end{theorem}

Equations~\eqref{eq:TanConeRealGenericDeterminantalVarietyOrthogonal} and~\eqref{eq:TanConeRealGenericDeterminantalVariety} are respectively~\cite[(2.5) and (2.6)]{HosseiniLukeUschmajew2019}. In view of the first equality in~\eqref{eq:NorSpaceFixedRankManifoldTensorProduct} below,~\eqref{eq:TanConeRealGenericDeterminantalVarietyOrthogonal} is equivalent to the statement of~\cite[Theorem~3.2]{SchneiderUschmajew2015}; the proof therein is discussed in section~\ref{sec:GapProofSU15Thm3.2} and completed in section~\ref{sec:Proof}.

In the rest of this section, given $\ushort{r} \in \{0, \dots, \min\{m, n\}\}$ and $X \in \R_{\ushort{r}}^{m \times n}$, the tangent and normal spaces to $\R_{\ushort{r}}^{m \times n}$ at $X$ are expressed as functions of $\im X$, $\im X^\tp$, and their orthogonal complements by relying on orthonormal bases or tensor products of these subspaces. These expressions enable to reformulate Theorem~\ref{thm:TanConeRealGenericDeterminantalVariety} via matrix equations in Theorem~\ref{thm:TanConeRealGenericDeterminantalVarietyMatrix}.

Let $U \in \st_\R(\ushort{r}, m)$, $U_\perp \in \st_\R(m-\ushort{r}, m)$, $V \in \st_\R(\ushort{r}, n)$, and $V_\perp \in \st_\R(n-\ushort{r}, n)$ such that $\im U = \im X$, $\im U_\perp = (\im X)^\perp$, $\im V = \im X^\tp$, and $\im V_\perp = (\im X^\tp)^\perp$.
Then,
\begin{align}
\norcone{\R_{\ushort{r}}^{m \times n}}{X}
\label{eq:NorSpaceFixedRankManifoldMatrix}
&= [U \; U_\perp] \begin{bmatrix} 0_{\ushort{r} \times \ushort{r}} & 0_{\ushort{r} \times n-\ushort{r}} \\ 0_{m-\ushort{r} \times \ushort{r}} & \R^{m-\ushort{r} \times n-\ushort{r}} \end{bmatrix} [V \; V_\perp]^\tp = U_\perp \R^{m-\ushort{r} \times n-\ushort{r}} V_\perp^\tp,\\
\tancone{\R_{\ushort{r}}^{m \times n}}{X}
\label{eq:TanSpaceFixedRankManifoldMatrix}
&= [U \; U_\perp] \begin{bmatrix} \R^{\ushort{r} \times \ushort{r}} & \R^{\ushort{r} \times n-\ushort{r}} \\ \R^{m-\ushort{r} \times \ushort{r}} & 0_{m-\ushort{r} \times n-\ushort{r}} \end{bmatrix} [V \; V_\perp]^\tp.
\end{align}
Equation~\eqref{eq:TanSpaceFixedRankManifoldMatrix} is given in~\cite[Proposition~2.1]{Vandereycken2013}. Equation~\eqref{eq:NorSpaceFixedRankManifoldMatrix} follows from~\eqref{eq:TanSpaceFixedRankManifoldMatrix}. By~\eqref{eq:TensorProductMatrixFactorization}, we can rewrite~\eqref{eq:NorSpaceFixedRankManifoldMatrix} and~\eqref{eq:TanSpaceFixedRankManifoldMatrix} as
\begin{align}
\norcone{\R_{\ushort{r}}^{m \times n}}{X}
\label{eq:NorSpaceFixedRankManifoldTensorProduct}
&= (\im X)^\perp \otimes (\im X^\tp)^\perp
= \ker X^\tp \otimes \ker X,\\
\tancone{\R_{\ushort{r}}^{m \times n}}{X}
\label{eq:TanSpaceFixedRankManifoldTensorProduct}
&= \left((\im X) \otimes (\im X^\tp)\right) + \left((\im X)^\perp \otimes (\im X^\tp)\right) + \left((\im X) \otimes (\im X^\tp)^\perp\right).
\end{align}
The three terms in the right-hand side of~\eqref{eq:TanSpaceFixedRankManifoldTensorProduct} are mutually orthogonal linear subspaces of $\R^{m \times n}$. Furthermore, by~\eqref{eq:NorSpaceFixedRankManifoldMatrix} and~\eqref{eq:TanSpaceFixedRankManifoldMatrix}, for all $Z \in \R^{m \times n}$,
\begin{align}
\proj{\norcone{\R_{\ushort{r}}^{m \times n}}{X}}{Z}
\label{eq:ProjNorSpaceFixedRankManifold}
&= [U \; U_\perp] \begin{bmatrix} 0_{\ushort{r} \times \ushort{r}} & 0_{\ushort{r} \times n-\ushort{r}} \\[1mm] 0_{m-\ushort{r} \times \ushort{r}} & U_\perp^\tp Z V_\perp \end{bmatrix} [V \; V_\perp]^\tp\\
\nonumber
&= U_\perp U_\perp^\tp Z V_\perp V_\perp^\tp,\\
\proj{\tancone{\R_{\ushort{r}}^{m \times n}}{X}}{Z}
\label{eq:ProjTanSpaceFixedRankManifold}
&= [U \; U_\perp] \begin{bmatrix} U^\tp Z V & U^\tp Z V_\perp \\ U_\perp^\tp Z V & 0_{m-\ushort{r} \times n-\ushort{r}} \end{bmatrix} [V \; V_\perp]^\tp\\
\nonumber
&= UU^\tp Z + Z VV^\tp - UU^\tp Z VV^\tp.
\end{align}
Equation~\eqref{eq:ProjTanSpaceFixedRankManifold} is used in section~\ref{subsec:LinkAlgebraGeometry}.

\begin{theorem}
\label{thm:TanConeRealGenericDeterminantalVarietyMatrix}
Let $X \in \R_{\le r}^{m \times n}$ and $\ushort{r} \coloneq \rank X$. Let $U \in \st_\R(\ushort{r}, m)$, $U_\perp \in \st_\R(m-\ushort{r}, m)$, $V \in \st_\R(\ushort{r}, n)$, and $V_\perp \in \st_\R(n-\ushort{r}, n)$ such that $\im U = \im X$, $\im U_\perp = (\im X)^\perp$, $\im V = \im X^\tp$, and $\im V_\perp = (\im X^\tp)^\perp$. Then,
\begin{align*}
\tancone{\R_{\le r}^{m \times n}}{X}
&= [U \; U_\perp] \begin{bmatrix} \R^{\ushort{r} \times \ushort{r}} & \R^{\ushort{r} \times n-\ushort{r}} \\ \R^{m-\ushort{r} \times \ushort{r}} & \R_{\le r-\ushort{r}}^{m-\ushort{r} \times n-\ushort{r}} \end{bmatrix} [V \; V_\perp]^\tp\\
&= \left\{Z \in \R^{m \times n} \mid \rank U_\perp^\tp Z V_\perp \le r-\ushort{r}\right\}.
\end{align*}
\end{theorem}

\begin{proof}
By~\eqref{eq:NorSpaceFixedRankManifoldMatrix} and~\eqref{eq:TanSpaceFixedRankManifoldMatrix}, the first equality is equivalent to~\eqref{eq:TanConeRealGenericDeterminantalVarietyOrthogonal} since
\begin{equation*}
U_\perp \R^{m-\ushort{r} \times n-\ushort{r}} V_\perp^\tp \cap \R_{\le r-\ushort{r}}^{m \times n} = U_\perp \R_{\le r-\ushort{r}}^{m-\ushort{r} \times n-\ushort{r}} V_\perp^\tp.
\end{equation*}
By~\eqref{eq:ProjNorSpaceFixedRankManifold}, the second equality is equivalent to~\eqref{eq:TanConeRealGenericDeterminantalVarietyNormal} since, for all $Z \in \R^{m \times n}$,
\begin{equation*}
\rank U_\perp U_\perp^\tp Z V_\perp V_\perp^\tp = \rank U_\perp^\tp Z V_\perp;
\end{equation*}
the inequality $\le$ is clear, and the inequality $\ge$ holds as $U_\perp^\tp Z V_\perp = U_\perp^\tp (U_\perp U_\perp^\tp Z V_\perp V_\perp^\tp) V_\perp$.
\end{proof}

\section{Proof of Theorem~\ref{thm:TanConeRealGenericDeterminantalVariety}}
\label{sec:Proof}
The novel part consists of section~\ref{subsec:ProofFirstInclusion}. The rest of the proof is straightforward or known, and is given for completeness.
If $X = 0_{m \times n}$, then Theorem~\ref{thm:TanConeRealGenericDeterminantalVariety} follows from Proposition~\ref{prop:TanConeToClosedConeAtTheOrigin}  since $\R_{\le r}^{m \times n}$ is a closed cone. Assume for the rest of the section that $X \ne 0_{m \times n}$.
The equality
\begin{equation}
\label{eq:TanConeRealGenericDeterminantalVarietyOrthogonalNormal}
\tancone{\R_{\ushort{r}}^{m \times n}}{X} + \left(\norcone{\R_{\ushort{r}}^{m \times n}}{X} \cap \R_{\le r-\ushort{r}}^{m \times n}\right) = \left\{Z \in \R^{m \times n} \mid \rank \proj{\norcone{\R_{\ushort{r}}^{m \times n}}{X}}{Z} \le r-\ushort{r}\right\}
\end{equation}
holds because, for all $Z \in \R^{m \times n}$, $Z = \proj{\tancone{\R_{\ushort{r}}^{m \times n}}{X}}{Z} + \proj{\norcone{\R_{\ushort{r}}^{m \times n}}{X}}{Z}$.
Thus, to prove Theorem~\ref{thm:TanConeRealGenericDeterminantalVariety}, it suffices to establish the inclusions
\begin{equation}
\label{eq:TanConeRealGenericDeterminantalVarietyThreeInclusions}
\tancone{\R_{\le r}^{m \times n}}{X}
\: \subseteq\: \tancone{\R_{\ushort{r}}^{m \times n}}{X} + \left(\norcone{\R_{\ushort{r}}^{m \times n}}{X} \cap \R_{\le r-\ushort{r}}^{m \times n}\right)
\: \subseteq\: \tancone{\R_{\ushort{r}}^{m \times n}}{X} + \R_{\le r-\ushort{r}}^{m \times n}
\: \subseteq\: \tancone{\R_{\le r}^{m \times n}}{X}.
\end{equation}
The second inclusion is clear. The first and third inclusions are established in sections~\ref{subsec:ProofFirstInclusion} and~\ref{subsec:ProofThirdInclusion}, respectively.

\subsection{Proof of the first inclusion in \texorpdfstring{\eqref{eq:TanConeRealGenericDeterminantalVarietyThreeInclusions}}{(12)}}
\label{subsec:ProofFirstInclusion}

\subsubsection{Preliminaries}
Let $Z \in \tancone{\R_{\le r}^{m \times n}}{X}$. By~\eqref{eq:TanConeRealGenericDeterminantalVarietyOrthogonalNormal}, it suffices to prove that
\begin{equation*}
\rank \proj{\norcone{\R_{\ushort{r}}^{m \times n}}{X}}{Z} \le r-\ushort{r}.
\end{equation*}
By the proof of the second equality in Theorem~\ref{thm:TanConeRealGenericDeterminantalVarietyMatrix}, $\rank \proj{\norcone{\R_{\ushort{r}}^{m \times n}}{X}}{Z} = \rank U_\perp^\tp Z V_\perp$. By definition, there exist sequences $(X_i)_{i \in \N}$ in $\R_{\le r}^{m \times n}$ converging to $X$ and $(t_i)_{i \in \N}$ in $(0, \infty)$ such that $(\frac{X_i-X}{t_i})_{i \in \N}$ converges to~$Z$. The inequality $\rank U_\perp^\tp Z V_\perp \le r-\ushort{r}$ is established in section~\ref{subsubsec:MainInequality}. Other proofs are proposed in section~\ref{sec:DiscussionsComplements}.

\subsubsection{Main inequality}
\label{subsubsec:MainInequality}
For every $i \in \N$, let $\oshort{U}_{i\perp} \in \st_\R(m-r, m)$ such that $\im \oshort{U}_{i\perp} \subseteq (\im X_i)^\perp$. Then, for all $i \in \N$,
\begin{align*}
\oshort{U}_{i\perp}^\tp X_i = 0_{m-r \times n},&&
\oshort{U}_{i\perp}^\tp \frac{X_i-X}{t_i} V_\perp = 0_{m-r \times n-\ushort{r}}.
\end{align*}
By the Bolzano--Weierstrass theorem, since $\st_\R(m-r, m)$ is closed and bounded, a subsequence $(\oshort{U}_{i_k\perp})_{k \in \N}$ converges to $\oshort{U}_\perp \in \st_\R(m-r, m)$. Thus,
\begin{align}
\label{eq:LimitEquality}
\oshort{U}_\perp^\tp X = 0_{m-r \times n},&&
\oshort{U}_\perp^\tp Z V_\perp = 0_{m-r \times n-\ushort{r}}.
\end{align}
The first equality in~\eqref{eq:LimitEquality} implies $\im \oshort{U}_\perp \subseteq (\im X)^\perp$. Hence, since $(\im X)^\perp = \im U_\perp$, it holds that $\oshort{U}_\perp = U_\perp U_\perp^\tp \oshort{U}_\perp$. Therefore, $U_\perp^\tp \oshort{U}_\perp \in \st_\R(m-r, m-\ushort{r})$ and the second equality in~\eqref{eq:LimitEquality} yields $\oshort{U}_\perp^\tp U_\perp (U_\perp^\tp Z V_\perp) = 0_{m-r \times n-\ushort{r}}$. Thus, $\im U_\perp^\tp Z V_\perp \subseteq (\im U_\perp^\tp \oshort{U}_\perp)^\perp$. Hence,
\begin{equation*}
\rank U_\perp^\tp Z V_\perp
\le \dim ((\im U_\perp^\tp \oshort{U}_\perp)^\perp)
= (m-\ushort{r}) - \rank U_\perp^\tp \oshort{U}_\perp
= (m-\ushort{r}) - (m-r)
= r-\ushort{r}.
\end{equation*}

\subsection{Proof of the third inclusion in \texorpdfstring{\eqref{eq:TanConeRealGenericDeterminantalVarietyThreeInclusions}}{(12)}}
\label{subsec:ProofThirdInclusion}
This section merely rephrases the first part of~\cite[proof of Theorem~3.2]{SchneiderUschmajew2015}. Let $Z \in \tancone{\R_{\ushort{r}}^{m \times n}}{X} + \R_{\le r-\ushort{r}}^{m \times n}$. Then, $Z = Z_1 + Z_2$ with $Z_1 \in \tancone{\R_{\ushort{r}}^{m \times n}}{X}$ and $Z_2 \in \R_{\le r-\ushort{r}}^{m \times n}$. Hence, there exist sequences $(Y_i)_{i \in \N}$ in $\R_{\ushort{r}}^{m \times n}$ converging to $X$ and $(t_i)_{i \in \N}$ in $(0, \infty)$ such that $(\frac{Y_i-X}{t_i})_{i \in \N}$ converges to $Z_1$. For all $i \in \N$, $X_i \coloneq Y_i+t_iZ_2 \in \R_{\le r}^{m \times n}$ and $\frac{X_i-X}{t_i} = \frac{Y_i-X}{t_i}+Z_2 \xrightarrow{i \to \infty} Z$. Therefore, $Z \in \tancone{\R_{\le r}^{m \times n}}{X}$.

\section{The gap in \texorpdfstring{\cite[proof of Theorem~3.2]{SchneiderUschmajew2015}}{Theorem 3.2 in Schneider/Uschmajew 2015}}
\label{sec:GapProofSU15Thm3.2}
This section presents an example of a sequence $(X_i)_{i \in \N \setminus \{0\}}$ in $\R_{\le r}^{m \times n}$ converging to $X \in \R_{\ushort{r}}^{m \times n}$ such that $(i(X_i-X))_{i \in \N \setminus \{0\}}$ converges and $\rank \proj{\norcone{\R_{\ushort{r}}^{m \times n}}{X}}{X_i} > r-\ushort{r}$ for all $i \in \N \setminus \{0\}$. This shows that the proof of the first inclusion in~\eqref{eq:TanConeRealGenericDeterminantalVarietyThreeInclusions} proposed in \cite[proof of Theorem~3.2]{SchneiderUschmajew2015} is not entirely correct since, for all $i \in \N \setminus \{0\}$, $\proj{\norcone{\R_{\ushort{r}}^{m \times n}}{X}}{i(X_i-X)} = i \proj{\norcone{\R_{\ushort{r}}^{m \times n}}{X}}{X_i}$ has rank larger than $r - \ushort{r}$. Define $m \coloneq n \coloneq 4$, $r \coloneq 3$, $\ushort{r} \coloneq 2$,
\begin{align*}
X \coloneq
\begin{bmatrix}
1 & 0 & 0 & 0\\
0 & 1 & 0 & 0\\
0 & 0 & 0 & 0\\
0 & 0 & 0 & 0
\end{bmatrix}, \quad \text{and} \quad
X_i \coloneq
\begin{bmatrix}
1 & 0 & 0 & 0\\
0 & 1 & 0 & \frac{1}{i}\\
0 & 0 & \frac{1}{i} & 0\\
0 & \frac{1}{i} & 0 & \frac{1}{i^2}
\end{bmatrix}
\text{ for all } i \in \N \setminus \{0\}.
\end{align*}
Then, $(X_i)_{i \in \N \setminus \{0\}}$ converges to $X$ and
\begin{equation*}
\lim_{i \to \infty} i (X_i - X) =
\begin{bmatrix}
0 & 0 & 0 & 0\\
0 & 0 & 0 & 1\\
0 & 0 & 1 & 0\\
0 & 1 & 0 & 0
\end{bmatrix}
\in \tancone{\R_{\le r}^{m \times n}}{X}.
\end{equation*}
By~\eqref{eq:ProjNorSpaceFixedRankManifold}, for all $i \in \N \setminus \{0\}$,
\begin{equation*}
\proj{\norcone{\R_{\ushort{r}}^{m \times n}}{X}}{X_i} =
\begin{bmatrix}
0 & 0 & 0 & 0\\
0 & 0 & 0 & 0\\
0 & 0 & \frac{1}{i} & 0\\
0 & 0 & 0 & \frac{1}{i^2}
\end{bmatrix}.
\end{equation*}
Thus, for all $i \in \N \setminus \{0\}$, $\rank \proj{\norcone{\R_{\ushort{r}}^{m \times n}}{X}}{X_i} = 2 > 1$.

Though, as we have just seen, $\rank \proj{\norcone{\R_{\ushort{r}}^{m \times n}}{X}}{X_i} > r-\ushort{r}$ for all $i \in \N \setminus \{0\}$, it does hold that $\rank (X_i - L_i) \leq r-\ushort{r}$, where $L_i$ is a nearest point to $X+\proj{\tancone{\R_{\ushort{r}}^{m \times n}}{X}}{X_i-X}$ in $(X+ \proj{\tancone{\R_{\ushort{r}}^{m \times n}}{X}}{X_i-X} +\norcone{\R_{\ushort{r}}^{m \times n}}{X})\cap\R_{\ushort{r}}^{m \times n}$. This fact is proven and exploited in section~\ref{sec:DiscussionsComplements}, where an alternative proof to that from section~\ref{subsubsec:MainInequality} is given.

\section{The algebraic tangent cone to the generic determinantal variety}
\label{sec:AlgebraicTangentConeGenericDeterminantalVariety}
This section focuses on the \emph{algebraic} tangent cone to the generic determinantal variety over an algebraically closed field or $\R$.
Section~\ref{subsec:ElementsAlgebraicAnalyticGeometry} reviews basic definitions and results from algebraic and analytic geometry. Then, it is deduced in section~\ref{subsec:AlgebraicTangentConeRealGenericDeterminantalVariety} that, for the generic determinantal variety over~$\R$, the algebraic tangent cone and the tangent cone coincide. Finally, in section~\ref{subsec:ExpressionViaGrassmann}, the expression for the algebraic tangent cone to $F_{\le r}^{m \times n}$ given in~\cite[(2.2)]{ArbarelloCornalbaGriffithsHarris} for $F = \C$ is shown to be equivalent to that of the tangent cone to that variety for $F \in \{\R, \C\}$.

\subsection{Elements of algebraic and analytic geometry}
\label{subsec:ElementsAlgebraicAnalyticGeometry}
This section can readily be skipped by readers familiar with the concept of algebraic tangent cone.

\subsubsection{Affine algebraic sets}
\label{subsubsec:AffineAlgebraicSets}
Let $F$ be a field. A subset of $F^n$ is called an \emph{affine algebraic set} if it is the set of common roots of a set of polynomials with coefficients in $F$~\cite[Definition~1.2.1]{Mangolte}. Without loss of generality, the set of polynomials can be assumed to be finite and even to be a singleton if $F = \R$. This definition is also given in~\cite[Chapter~II]{AkbulutKing},~\cite[Definition~2.1.1]{BochnakCosteRoy} for $F$ a real closed field,~\cite[\S 7.1]{Chirka} for $F = \C$, and~\cite[\S 1]{GhomiHoward} for $F = \R$. An affine algebraic set is called an \emph{affine variety} in~\cite[Chapter~1, \S 2]{CoxLittleOShea} and, for $F$ algebraically closed, in~\cite[Lecture~1]{Harris}. In contrast, the term ``affine algebraic variety'' names an affine algebraic set ``up to a biregular isomorphism'' in~\cite[Definitions~1.3.1 and 1.3.10]{Mangolte} and, for $F$ a real closed field, in~\cite[Definition~3.2.9]{BochnakCosteRoy}.

The adjective ``affine'' is used for distinction with \emph{projective} algebraic sets and varieties, which are defined in all the books cited in this section.
The set $F_{\le r}^{m \times n}$, which is an affine algebraic subset of $F^{m \times n}$, can also be viewed as a projective algebraic set since determinants are homogeneous polynomials. However, this is not the point of view chosen in this work.

\subsubsection{Real and complex analytic sets}
\label{subsubsec:RealComplexAnalyticSets}
Let $F$ be $\R$ or $\C$. A subset of $F^n$ is said to be \emph{analytic} if it is locally the set of common zeros of a set of analytic functions from $F^n$ to $F$; see~\cite[Definition~D.1.1]{Mangolte} for $F = \C$ and~\cite[\S 1]{GhomiHoward} for $F = \R$. Without loss of generality, the set of analytic functions can be assumed to be finite if $F = \C$~\cite[\S 3]{Whitney1965} and to be a singleton if $F = \R$. An analytic subset of $\C^n$ is called a \emph{complex analytic variety} in~\cite{Whitney1965Chapter,Whitney1965,Kendig1972}. In contrast, the term ``complex analytic variety'' names a more general object in~\cite[Definition~D.2.1]{Mangolte}.

Analytic subsets of $\R^n$ can have a pathological global behavior, which motivates the following definition: a subset of $\R^n$ is said to be \emph{$\C$-analytic} if it is the real part of an analytic subset of $\C^n$~\cite{WhitneyBruhat}. This definition also appears in~\cite{Kendig1972}. Affine algebraic subsets of $\R^n$ are $\C$-analytic~\cite[\S 10]{WhitneyBruhat}.

\subsubsection{Tangent cones to real and complex affine algebraic sets}
\label{subsubsec:TangentConesRealComplexAffineAlgebraicSets}
This section compares several notions of tangent cone for real and complex affine algebraic sets. Let $F$ be a field and $F[X_1, \dots, X_n]$ the polynomial ring in $n$ indeterminates over $F$~\cite[Chapter~1, \S 1]{CoxLittleOShea}. For every $f \in F[X_1, \dots, X_n]$ and $x \in F^n$, there exists $d \in \N$ such that $f = \sum_{i=0}^d f_{x, i}$, where each $f_{x, i}$ is a linear combination of $\prod_{j=1}^n (X_j-x_j)^{k_j}$ with $k_j \in \N$ and $\sum_{j=1}^n k_j = i$. If $f \ne 0$, then there exists a smallest $i \in \{0, \dots, d\}$ such that $f_{x, i} \ne 0$, and $f_{x, i}$ is called the \emph{initial polynomial} of $f$ at $x$.

Let $A$ be an affine algebraic subset of $F^n$ and $x \in A$.
The \emph{algebraic tangent cone} to $A$ at~$x$ is the set of common roots of the initial polynomials at $x$ of all polynomials that are zero on~$A$; see~\cite[Chapter~9, \S 7, Definition~2]{CoxLittleOShea} and, for $F \in \{\R, \C\}$,~\cite[\S 10]{Whitney1965},~\cite[Notation~1.6]{Kendig1972},~\cite[\S 2]{OSheaWilson2004}, and~\cite[\S 8.4]{Chirka}. Thus, the algebraic tangent cone to $A$ at $x$ is an affine algebraic subset of~$F^n$.

Let $F$ be $\R$ or $\C$. A vector $v \in F^n$ is said to be \emph{tangent} to $A$ at $x$ if there exist sequences $(x_i)_{i \in \N}$ in $A$ converging to $x$ and $(\alpha_i)_{i \in \N}$ in $(0, \infty)$ such that $(\alpha_i(x-x_i))_{i \in \N}$ converges to $v$; this definition is that given in section~\ref{sec:ProjectionTangentCone} if $F = \R$ and comes from~\cite[\S 8]{Whitney1965} if $F = \C$. The \emph{tangent cone} to $A$ at $x$, denoted $\tancone{A}{x}$, is the set of all tangent vectors to $A$ at $x$.

For $F = \C$, allowing $(\alpha_i)_{i \in \N}$ to be in $\C$ does not change the defined cone~\cite[Remark~8.2]{Whitney1965}. Moreover, the algebraic tangent cone and the tangent cone coincide~\cite[Theorem~10.6]{Whitney1965}; see also~\cite[Notation~1.6]{Kendig1972},~\cite[\S 2]{OSheaWilson2004},~\cite[\S 8.4]{Chirka}, and~\cite[Chapter~9, \S 7, Theorem~6]{CoxLittleOShea}.

For $F = \R$, the cone obtained by allowing $(\alpha_i)_{i \in \N}$ to be in $\R$ is called the \emph{symmetric tangent cone} in~\cite[\S 2.1]{GhomiHoward} and the \emph{geometric tangent cone} in~\cite[Definition~1.3]{Kendig1972} and~\cite[\S 2]{OSheaWilson2004}; it is $\tancone{A}{x} \cup (-\tancone{A}{x})$. The tangent cone is called the \emph{tangent semicone} in~\cite[\S 2]{OSheaWilson2004}. The symmetric tangent cone is included in the algebraic tangent cone, and the inclusion can be strict~\cite{Kendig1972,OSheaWilson2004}.

Tangent vectors to $A$ can be realized as tangents to analytic arcs contained in $A$:
\begin{equation}
\label{eq:TangentVectorsToAlgebraicSetsAsTangentsToAnalyticArcsInTheSet}
\tancone{A}{x} = \left\{\gamma'(0) \mid \gamma : [0, \varepsilon) \to F^n \text{ is analytic with } \varepsilon \in (0, \infty),\, \gamma(0) = x,\, \gamma([0, \varepsilon)) \subseteq A\right\}.
\end{equation}
The inclusion $\supseteq$ is clear. The inclusion $\subseteq$ is stated in~\cite[\S 8.5, Proposition~2]{Chirka} for $F = \C$ and in~\cite[Proposition~2]{OSheaWilson2004} for $F = \R$.

\subsection{The algebraic tangent cone to the real generic determinantal variety}
\label{subsec:AlgebraicTangentConeRealGenericDeterminantalVariety}
The algebraic tangent cone to the generic determinantal variety over $\R$ seems to be unavailable in the literature, which focuses on algebraically closed fields. However, the derivation of the algebraic tangent cone to $F_{\le r}^{m \times n}$ proposed in~\cite[Example~20.5]{Harris} for $F$ an algebraically closed field is also valid for $F = \R$. The resulting expression for the algebraic tangent cone to $\R_{\le r}^{m \times n}$ turns out to coincide with those of the tangent cone to $\R_{\le r}^{m \times n}$ reviewed in section~\ref{sec:TangentConeRealGenericDeterminantalVariety}; this coincidence does not hold for all real algebraic varieties, as recalled in section~\ref{subsubsec:TangentConesRealComplexAffineAlgebraicSets}. This is recorded in Theorem~\ref{thm:AlgebraicTangentConeGenericDeterminantalVariety}. Nevertheless, we do not known how to deduce the tangent cone to $\R_{\le r}^{m \times n}$ from the algebraic tangent cone to $\R_{\le r}^{m \times n}$.

\begin{theorem}
\label{thm:AlgebraicTangentConeGenericDeterminantalVariety}
Let $F$ be an algebraically closed field or $\R$. Let $X \in F_{\le r}^{m \times n}$ and $\ushort{r} \coloneq \rank X$. Let $P \in F_{m-\ushort{r}}^{m \times m-\ushort{r}}$ and $Q \in F_{n-\ushort{r}}^{n \times n-\ushort{r}}$ such that $\im P = \ker X^\tp$ and $\im Q = \ker X$. Then:
\begin{enumerate}
\item the algebraic tangent cone to $F_{\le r}^{m \times n}$ at $X$ is
\begin{equation*}
\left\{Z \in F^{m \times n} \mid \rank P^\tp Z Q \le r-\ushort{r}\right\};
\end{equation*}
\item $\tancone{\R_{\le r}^{m \times n}}{X}$ equals the algebraic tangent cone to $\R_{\le r}^{m \times n}$ at $X$.
\end{enumerate}
\end{theorem}

\begin{proof}
The second statement follows from the first and Theorem~\ref{thm:TanConeRealGenericDeterminantalVarietyMatrix}. Let us establish the first statement.

Let $F$ be algebraically closed.
There exist $\hat{P} \in F_m^{m \times m}$ and $\hat{Q} \in F_n^{n \times n}$ such that $\hat{P}X\hat{Q} = \mathrm{diag}(I_{\ushort{r}}, 0_{m-\ushort{r} \times n-\ushort{r}})$. By~\cite[Example~20.5]{Harris}, the algebraic tangent cone to $F_{\le r}^{m \times n}$ at $X$ is
\begin{equation*}
\left\{Z \in F^{m \times n} \mid \rank \hat{P}_{\ushort{r}+1 \mathord{:} m, \mathord{:}} Z \hat{Q}_{\mathord{:}, \ushort{r}+1 \mathord{:} n} \le r-\ushort{r}\right\}.
\end{equation*}
Furthermore, $\im \, (\hat{P}_{\ushort{r}+1 \mathord{:} m, \mathord{:}})^\tp = \ker X^\tp$ and $\im \hat{Q}_{\mathord{:}, \ushort{r}+1 \mathord{:} n} = \ker X$. 
Thus, there exist $R \in F_{m-\ushort{r}}^{m-\ushort{r} \times m-\ushort{r}}$ and $S \in F_{n-\ushort{r}}^{n-\ushort{r} \times n-\ushort{r}}$ such that $P = (\hat{P}_{\ushort{r}+1 \mathord{:} m, \mathord{:}})^\tp R$ and $Q = \hat{Q}_{\mathord{:}, \ushort{r}+1 \mathord{:} n} S$. Hence, for all $Z \in F^{m \times n}$, $\rank \hat{P}_{\ushort{r}+1 \mathord{:} m, \mathord{:}} Z \hat{Q}_{\mathord{:}, \ushort{r}+1 \mathord{:} n} = \rank P^\tp Z Q$.

The argument given in~\cite[Example~20.5]{Harris} also holds for $F = \R$ since it does not rely on the algebraic closedness of $F$. The only part that may require a verification is the irreducibility. Again, the argument given in~\cite[Example~12.1]{Harris} also holds for $F = \R$ since it does not rely on the algebraic closedness of $F$. Thus, the generic determinantal variety over $\R$ is irreducible, as also stated in~\cite[\S III.B]{Tsakiris2023} without proof and in~\cite[\S 2.1.2]{YaoPengTsakiris} where the given reference seems to focus on algebraically closed fields.
\end{proof}

\subsection{An expression based on the Grassmann manifold}
\label{subsec:ExpressionViaGrassmann}
Let $F$ be $\R$ or $\C$. For every vector space $E$ over $F$ and every $p \in \N$, the set $\grass(p, E)$ of all $p$-dimensional linear subspaces of $E$ is a smooth manifold called a \emph{Grassmann manifold}; see, e.g.,~\cite[\S 3.4.4]{AbsilMahonySepulchre} for $F = \R$. For every $X \in F_{\le r}^{m \times n}$,
\begin{equation}
\label{eq:TanConeComplexGenericDeterminantalVariety}
\tancone{F_{\le r}^{m \times n}}{X} = \left\{Z \in F^{m \times n} \mid \exists S \in \grass(n-r, \ker X) : ZS \subseteq \im X\right\},
\end{equation}
where $ZS$ denotes the image of $S$ under $Z$. For $F = \C$, this is~\cite[(2.2)]{ArbarelloCornalbaGriffithsHarris}. In this section, this expression is shown to be equivalent to that given in Theorem~\ref{thm:AlgebraicTangentConeGenericDeterminantalVariety}.

Let $\ushort{r} \coloneq \rank X$. Let $U \in \st_F(\ushort{r}, m)$, $U_\perp \in \st_F(m-\ushort{r}, m)$, $V \in \st_F(\ushort{r}, n)$, and $V_\perp \in \st_F(n-\ushort{r}, n)$ such that $\im U = \im X$, $\im U_\perp = (\im X)^\perp$, $\im V = \im X^*$, and $\im V_\perp = (\im X^*)^\perp$. Let $Z \in F^{m \times n}$. Then, $Z = [U \; U_\perp] \begin{bmatrix} U^*ZV & U^*ZV_\perp \\ U_\perp^*ZV & U_\perp^*ZV_\perp \end{bmatrix} [V \; V_\perp]^*$. The following equivalences hold:
\begin{align*}
&\exists S \in \grass(n-r, \ker X) \text{ such that } ZS \subseteq \im X\\
&\text{iff }
\exists W \in \st_F(n-r, n) \text{ such that } \im W \subseteq \ker X \text{ and } \im ZW \subseteq \im X\\
&\text{iff }
\exists Y \in \st_F(n-r, n-\ushort{r}) \text{ such that } \im Z V_\perp Y \subseteq \im U\\
&\text{iff }
\exists Y \in \st_F(n-r, n-\ushort{r}) \text{ such that } U_\perp^* Z V_\perp Y = 0_{m-\ushort{r} \times n-r}
\\
&\text{iff }
\exists Y \in \st_F(n-r, n-\ushort{r}) \text{ such that } \im Y \subseteq \ker U_\perp^* Z V_\perp\\
&\text{iff }
n-r \le \dim \ker U_\perp^* Z V_\perp\\
&\text{iff }
n-r \le n-\ushort{r} - \rank U_\perp^* Z V_\perp\\
&\text{iff }
\rank U_\perp^* Z V_\perp \le r-\ushort{r}.
\end{align*}
For $F = \C$, this is consistent with Theorem~\ref{thm:AlgebraicTangentConeGenericDeterminantalVariety} because $\rank U_\perp^* Z V_\perp = \rank P^\tp Z Q$ since $\im U_\perp = (\im X)^\perp = \ker X^* = \overline{\ker X^\tp} = \overline{\im P} = \im \overline{P}$ and $\im V_\perp = (\im X^*)^\perp = \ker X$.

\section{Conclusion}
\label{sec:Conclusion}
Given a field $F$, the set $F_{\le r}^{m \times n}$ is an affine algebraic subset of $F^{m \times n}$ called the generic determinantal variety over $F$. This paper has considered the tangent cone and the algebraic tangent cone to $\R_{\le r}^{m \times n}$. The concepts of tangent cone to a nonempty subset of a Euclidean vector space and of algebraic tangent cone to an affine algebraic set have been reviewed respectively in sections~\ref{sec:ProjectionTangentCone} and~\ref{subsubsec:TangentConesRealComplexAffineAlgebraicSets}.

The algebraic tangent cone to $F_{\le r}^{m \times n}$ with $F$ algebraically closed has been known since the eighties in the algebraic geometry literature; see, e.g.,~\cite[(2.2)]{ArbarelloCornalbaGriffithsHarris} and~\cite[Example~20.5]{Harris}. However, the case where $F = \R$ does not seem to have been considered. A first contribution of this paper is the observation that the derivation of the algebraic tangent cone to $F_{\le r}^{m \times n}$ proposed in~\cite[Example~20.5]{Harris} for $F$ algebraically closed is also valid for $F = \R$. This yields the first statement of Theorem~\ref{thm:AlgebraicTangentConeGenericDeterminantalVariety}.

The second and main contribution of this paper is about the tangent cone to $\R_{\le r}^{m \times n}$, which is explicitly described in~\cite[Theorem~3.2]{SchneiderUschmajew2015}. Having detected a gap in the proof therein, as detailed in section~\ref{sec:GapProofSU15Thm3.2}, and unaware of a complete proof, we have presented a proof in this paper, relying on the standard definition of the tangent cone in variational analysis, reviewed in section~\ref{sec:ProjectionTangentCone}. No proof based on algebraic geometry is known to us. The second statement of Theorem~\ref{thm:AlgebraicTangentConeGenericDeterminantalVariety} is the observation that the tangent cone and the algebraic tangent cone to $\R_{\le r}^{m \times n}$ coincide, which is not true for all real affine algebraic sets.

\appendix

\section{Discussions and complements}
\label{sec:DiscussionsComplements}
Sections~\ref{subsec:GeometricProof} and~\ref{subsec:AlgebraicProof} are each a substitute for section~\ref{subsubsec:MainInequality}. Those sections can be read independently. The link between the two proofs given in sections~\ref{subsec:GeometricProof} and~\ref{subsec:AlgebraicProof}, respectively, is drawn in section~\ref{subsec:LinkAlgebraGeometry}. Section~\ref{subsec:TangentConeIntersection} reviews four results in the literature through the lens of the intersection rule for the tangent cone.

In what follows, for every $Y \in \R^{m \times n}$ and $\rho \in (0, \infty)$, $\ball(Y, \rho) \coloneq \{\tilde{Y} \in \R^{m \times n} \mid \norm{\tilde{Y}-Y} < \rho\}$ is the open ball of center $Y$ and radius $\rho$ in $\R^{m \times n}$, and the singular values of $Y$ are denoted by $\sigma_1(Y) \ge \dots \ge \sigma_{\min\{m, n\}}(Y) \ge 0$.

\subsection{A geometric proof}
\label{subsec:GeometricProof}
This proof relies on the differential-geometric concept of orthographic retraction. The proof proposed in section~\ref{subsec:AlgebraicProof} can be seen as an algebraic translation, as explained in section~\ref{subsec:LinkAlgebraGeometry}.

Only the case where $Z \ne 0_{m \times n}$ has to be considered. The sequence $(\proj{\norcone{\R_{\ushort{r}}^{m \times n}}{X}}{\frac{X_i-X}{t_i}})_{i \in \N}$ converges to $\proj{\norcone{\R_{\ushort{r}}^{m \times n}}{X}}{Z}$. However, its elements can have a rank larger than $r-\ushort{r}$, as proven in section~\ref{sec:GapProofSU15Thm3.2}. The proof therefore decomposes $\proj{\norcone{\R_{\ushort{r}}^{m \times n}}{X}}{\frac{X_i-X}{t_i}}$ into two terms in $\norcone{\R_{\ushort{r}}^{m \times n}}{X}$, one converging to $0_{m \times n}$ and one of rank at most $r-\ushort{r}$. This decomposition relies on the orthographic retraction on $\R_{\ushort{r}}^{m \times n}$, denoted by $R^\mathrm{orth}$; see~\cite[\S 4.4]{AbsilMalick} and~\cite[\S 3.2]{AbsilOseledets2015}. There exists an open set $O_X \subseteq \tancone{\R_{\ushort{r}}^{m \times n}}{X}$ that contains $0_{m \times n}$ such that the map
\begin{equation*}
R^\mathrm{orth}_X : O_X \to \R_{\ushort{r}}^{m \times n} : Y \mapsto \proj{(X+Y+\norcone{\R_{\ushort{r}}^{m \times n}}{X})\cap\R_{\ushort{r}}^{m \times n}}{X+Y}
\end{equation*}
is well defined and continuously differentiable. Moreover, for all $Y \in O_X$, $R^\mathrm{orth}_X(Y)-X-Y \in \norcone{\R_{\ushort{r}}^{m \times n}}{X}$. The fact that $R^\mathrm{orth}$ is a retraction implies that
\begin{equation*}
\lim_{O_X \ni Y \to 0_{m \times n}} \frac{\norm{R^\mathrm{orth}_X(Y)-(X+Y)}}{\norm{Y}} = 0.
\end{equation*}
Since $(\proj{\tancone{\R_{\ushort{r}}^{m \times n}}{X}}{X_i-X})_{i \in \N}$ converges to $0_{m \times n}$, it is asymptotically contained in $O_X$. Therefore, for all $i \in \N$ large enough,
\begin{equation*}
L_i \coloneq R_X^\mathrm{orth}(\proj{\tancone{\R_{\ushort{r}}^{m \times n}}{X}}{X_i-X})
\end{equation*}
is well defined, $L_i-X-\proj{\tancone{\R_{\ushort{r}}^{m \times n}}{X}}{X_i-X} \in \norcone{\R_{\ushort{r}}^{m \times n}}{X}$, and
\begin{equation}
\label{eq:GeometricNormalDecomposition}
P_{\norcone{\R_{\ushort{r}}^{m \times n}}{X}}\left(\frac{X_i-X}{t_i}\right) = \frac{L_i-X-\proj{\tancone{\R_{\ushort{r}}^{m \times n}}{X}}{X_i-X}}{t_i} + \frac{X_i-L_i}{t_i},
\end{equation}
hence $X_i-L_i \in \norcone{\R_{\ushort{r}}^{m \times n}}{X}$.
The first term converges to $0_{m \times n}$: for all $i \in \N$ large enough,
\begin{align*}
& \frac{\left\|L_i-X-\proj{\tancone{\R_{\ushort{r}}^{m \times n}}{X}}{X_i-X}\right\|}{t_i}\\
&= \underbrace{\frac{\norm{X_i-X}}{t_i}}_{\xrightarrow{i \to \infty}~\norm{Z}} \underbrace{\frac{\left\|L_i-X-\proj{\tancone{\R_{\ushort{r}}^{m \times n}}{X}}{X_i-X}\right\|}{\left\|\proj{\tancone{\R_{\ushort{r}}^{m \times n}}{X}}{X_i-X}\right\|}}_{\xrightarrow{i \to \infty}~0 \text{ since } R^\mathrm{orth} \text{ is a retraction}} \underbrace{\frac{\left\|\proj{\tancone{\R_{\ushort{r}}^{m \times n}}{X}}{X_i-X}\right\|}{\|X_i-X\|}}_{\le 1}\\
&\xrightarrow{i \to \infty} 0.
\end{align*}
Hence, the second term $(X_i-L_i)/t_i$ converges to $\proj{\norcone{\R_{\ushort{r}}^{m \times n}}{X}}{Z}$ as $i$ tends to infinity. It remains to prove that $\rank(X_i-L_i) \le r-\ushort{r}$ for all $i \in \N$ large enough; the result will follow by lower semicontinuity of the rank.

Let us show that $\im L_i \cap (\im X)^\perp = \{0_{m \times 1}\}$ for all $i \in \N$ sufficiently large. In a nutshell, this follows from $L_i \xrightarrow{i \to \infty} X$ with $\rank L_i = \ushort{r} = \rank X$ for all $i \in \N$ sufficiently large. In detail, by contradiction, suppose not. Let $N_1$ be the (infinite) set of all $i \in \N$ such that the subspace $\im L_i \cap (\im X)^\perp$ does not reduce to $\{0_{m \times 1}\}$. There exist $(u_i)_{i \in N_1}$ in $(\im X)^\perp$ and $(v_i)_{i \in N_1}$ in $\R^n$ such that $u_i = L_i v_i$, $v_i \in (\ker L_i)^\perp = \im L_i^\tp$, and $\norm{u_i} = 1$. Thus, given a singular value decomposition $L_i = U_i \Sigma_i V_i^\tp$, there exists $\tilde{v}_i \in \R^{\ushort{r}}$ such that $v_i = V_i \tilde{v}_i$. Hence, $u_i = U_i\Sigma_i\tilde{v}_i$ and $\tilde{v}_i = \Sigma_i^{-1} U_i^\tp u_i$. Therefore, $\norm{v_i} = \norm{\tilde{v}_i} \le \norm{\Sigma_i^{-1}}_2 \norm{U_i^\tp u_i} \le \norm{\Sigma_i^{-1}}_2 \norm{U_i^\tp}_2 \norm{u_i} = \frac{1}{\sigma_{\ushort{r}}(L_i)}$, where $\norm{\cdot}_2$ is the spectral norm. By continuity of the singular values \cite[Theorem~2.6.4]{HornJohnson}, $\sigma_{\ushort{r}}(L_i) \xrightarrow{i \to \infty} \sigma_{\ushort{r}}(X) \ne 0$. Thus, $(v_i)_{i \in N_1}$ is bounded. By the Bolzano--Weierstrass theorem, there exist $N_2 \subseteq N_1$ and $N_3 \subseteq N_2$ such that $(u_i)_{i \in N_2}$ converges to $u \in (\im X)^\perp$ and $(v_i)_{i \in N_3}$ converges to $v \in \R^n$. By continuity of the norm, $\norm{u} = 1$. Furthermore, since $L_i \xrightarrow{i \to \infty} X$ and $u_i = L_i v_i$ for all $i \in N_3$, it holds that $u = X v$. Hence, $u \in \im X \cap (\im X)^\perp$. This set being identically $\{0_{m \times 1}\}$, we have $u = 0_{m \times 1}$, a contradiction with $\norm{u} = 1$.

A similar argument yields that $\im L_i^\top \cap (\im X^\top)^\perp = \{0_{n \times 1}\}$ for all $i \in \N$ sufficiently large.

The final piece of the proof relies on the following result.

\begin{lemma}  
\label{lemma:rankA+B}
Let $\cU_1$ and $\cU_2$ be linear subspaces of $\R^m$ such that $\cU_1 \cap \cU_2 = \{0_{m \times 1}\}$, $\cV_1$ and $\cV_2$ be linear subspaces of $\R^n$ such that $\cV_1 \cap \cV_2 = \{0_{n \times 1}\}$, $A_1 \in \cU_1 \otimes \cV_1$, and $A_2 \in \cU_2 \otimes \cV_2$. Then, $\rank(A_1+A_2) = \rank A_1 + \rank A_2$.
\end{lemma}

\begin{proof}
For all $i \in \{1, 2\}$, let $U_i \in \st_\R(\dim \cU_i, m)$ and $V_i \in \st_\R(\dim \cV_i, n)$ such that $\im U_i = \cU_i$ and $\im V_i = \cV_i$. By~\eqref{eq:TensorProductMatrixFactorization}, for all $i \in \{1, 2\}$, $A_i = U_i S_i V_i^\top$ with $S_i \coloneq U_i^\tp A_i V_i$ and $\rank S_i = \rank A_i$. Hence
\begin{equation*}
A_1 + A_2 = \begin{bmatrix} U_1 & U_2 \end{bmatrix}
\begin{bmatrix} S_1 & 0_{\dim \cU_1 \times \dim \cV_2} \\ 0_{\dim \cU_2 \times \dim \cV_1} & S_2 \end{bmatrix}
\begin{bmatrix} V_1^\top \\ V_2^\top \end{bmatrix},
\end{equation*}
where the first and last matrix factors have full rank. It follows that $\rank(A_1+A_1) = \rank S_1 + \rank S_2 = \rank A_1 + \rank A_2$.
\end{proof}

Clearly, $L_i \in (\im L_i) \otimes (\im L_i^\top)$. We have also seen that $X_i - L_i \in \norcone{\R_{\ushort{r}}^{m \times n}}{X}
= (\im X)^\perp \otimes (\im X^\tp)^\perp$. Finally, recall that $\im L_i \cap (\im X)^\perp = \{0_{m \times 1}\}$ and $\im L_i^\top \cap (\im X^\top)^\perp = \{0_{n \times 1}\}$ for all $i \in \N$ large enough. It thus follows from Lemma~\ref{lemma:rankA+B} that $\rank L_i + \rank(X_i - L_i) = \rank X_i$ for all $i \in \N$ large enough. Hence $\rank(X_i - L_i) = \rank X_i - \rank L_i \leq r - \ushort{r}$ for all $i \in \N$ large enough, which is what remained to be proven.

\subsection{An algebraic translation}
\label{subsec:AlgebraicProof}
Define $S \coloneq U^\tp X V \in \R_{\ushort{r}}^{\ushort{r} \times \ushort{r}}$, $\tilde{Z}_{1, 1} \coloneq U^\tp Z V$, $\tilde{Z}_{1, 2} \coloneq U^\tp Z V_\perp$, $\tilde{Z}_{2, 1} \coloneq U_\perp^\tp Z V$, and $\tilde{Z}_{2, 2} \coloneq U_\perp^\tp Z V_\perp$. Then,
\begin{align*}
X = [U \; U_\perp] \begin{bmatrix} S & 0_{\ushort{r} \times n-\ushort{r}} \\ 0_{m-\ushort{r} \times \ushort{r}} & 0_{m-\ushort{r} \times n-\ushort{r}} \end{bmatrix} [V \; V_\perp]^\tp,&&
Z = [U \; U_\perp] \begin{bmatrix} \tilde{Z}_{1, 1} & \tilde{Z}_{1, 2}\\ \tilde{Z}_{2, 1} & \tilde{Z}_{2, 2} \end{bmatrix} [V \; V_\perp]^\tp.
\end{align*}
It suffices to prove that $\rank \tilde{Z}_{2, 2} \le r-\ushort{r}$. For all $i \in \N$, define $A_i \coloneq U^\tp X_i V$, $C_i \coloneq U^\tp X_i V_\perp$, $D_i \coloneq U_\perp^\tp X_i V$, and $E_i \coloneq U_\perp^\tp X_i V_\perp$ and observe that
\begin{equation*}
X_i = [U \; U_\perp] \begin{bmatrix} A_i & C_i\\ D_i & E_i\end{bmatrix} [V \; V_\perp]^\tp.
\end{equation*}
Moreover, $(A_i)_{i \in \N}$, $(C_i)_{i \in \N}$, $(D_i)_{i \in \N}$, $(E_i)_{i \in \N}$, $(\frac{A_i-S}{t_i})_{i \in \N}$, $(\frac{C_i}{t_i})_{i \in \N}$, $(\frac{D_i}{t_i})_{i \in \N}$, and $(\frac{E_i}{t_i})_{i \in \N}$ respectively converge to $S$, $0_{\ushort{r} \times n-\ushort{r}}$, $0_{m-\ushort{r} \times \ushort{r}}$, $0_{m-\ushort{r} \times n-\ushort{r}}$, $\tilde{Z}_{1, 1}$, $\tilde{Z}_{1, 2}$, $\tilde{Z}_{2, 1}$, and $\tilde{Z}_{2, 2}$. Thus, there exists $i_* \in \N$ such that, for all integers $i \ge i_*$, $A_i \in \ball(S, \sigma_{\ushort{r}}(S))$. By the Eckart--Young theorem~\cite{EckartYoung1936}, $\ball(S, \sigma_{\ushort{r}}(S)) \subseteq \R_{\ushort{r}}^{\ushort{r} \times \ushort{r}}$. Therefore, for all integers $i \ge i_*$, $A_i$ is invertible and
\begin{align}
[U \; U_\perp]^\tp \frac{X_i-X}{t_i} [V \; V_\perp]
\label{eq:AlgebraicDecomposition}
=~&
\begin{bmatrix}
\frac{A_i-S}{t_i} & \frac{C_i}{t_i}\\[1mm]
\frac{D_i}{t_i} & 0_{m-\ushort{r} \times n-\ushort{r}}
\end{bmatrix}
+
\begin{bmatrix}
0_{\ushort{r} \times \ushort{r}} & 0_{\ushort{r} \times n-\ushort{r}}\\[1mm]
0_{m-\ushort{r} \times \ushort{r}} & \frac{D_i A_i^{-1} C_i}{t_i}
\end{bmatrix}\\
\nonumber
&+
\begin{bmatrix}
0_{\ushort{r} \times \ushort{r}} & 0_{\ushort{r} \times n-\ushort{r}}\\[1mm]
0_{m-\ushort{r} \times \ushort{r}} & \frac{E_i - D_i A_i^{-1} C_i}{t_i}
\end{bmatrix}.
\end{align}
By continuity of the matrix inversion,
\begin{equation*}
\lim_{i \to \infty} D_i A_i^{-1} C_i/t_i
= \lim_{i \to \infty} D_i/t_i \lim_{i \to \infty} A_i^{-1} \lim_{i \to \infty} C_i
= \tilde{Z}_{2, 1} S^{-1} 0_{\ushort{r} \times n-\ushort{r}}
= 0_{m-\ushort{r} \times n-\ushort{r}},
\end{equation*}
hence
\begin{equation*}
\lim_{i \to \infty} (E_i - D_i A_i^{-1} C_i)/t_i = \tilde{Z}_{2, 2}.
\end{equation*}
For all integers $i \ge i_*$, since $\rank X_i \le r$, $\rank A_i = \ushort{r}$, and
\begin{equation*}
[U \; U_\perp]^\tp X_i [V \; V_\perp]
=\!
\begin{bmatrix}
I_{\ushort{r}} & 0_{\ushort{r} \times m-\ushort{r}}\\
D_i A_i^{-1} & I_{m-\ushort{r}}
\end{bmatrix}
\!
\begin{bmatrix}
A_i & 0_{\ushort{r} \times n-\ushort{r}}\\
0_{m-\ushort{r} \times \ushort{r}} & E_i - D_i A_i^{-1} C_i
\end{bmatrix}
\!
\begin{bmatrix}
I_{\ushort{r}} & A_i^{-1} C_i\\
0_{n-\ushort{r} \times \ushort{r}} & I_{n-\ushort{r}}
\end{bmatrix},
\end{equation*}
it holds that
\begin{equation*}
\rank(E_i - D_i A_i^{-1} C_i) \le r-\ushort{r}.
\end{equation*}
By lower semicontinuity of the rank, it follows that $\rank \tilde{Z}_{2, 2} \le r-\ushort{r}$.

\subsection{Link between the geometric proof and its algebraic translation}
\label{subsec:LinkAlgebraGeometry}
For all $i \in \N$, by~\eqref{eq:ProjTanSpaceFixedRankManifold},
\begin{equation*}
\proj{\tancone{\R_{\ushort{r}}^{m \times n}}{X}}{X_i-X} = [U \; U_\perp] \begin{bmatrix} A_i-S & C_i\\ D_i & 0 _{m-\ushort{r} \times n-\ushort{r}} \end{bmatrix} [V \; V_\perp]^\tp,
\end{equation*}
hence the first term in~\eqref{eq:AlgebraicDecomposition} equals $[U \; U_\perp]^\tp \proj{\tancone{\R_{\ushort{r}}^{m \times n}}{X}}{\frac{X_i-X}{t_i}} [V \; V_\perp]$.
Moreover, for all $i \in \N$ large enough, by~\cite[Proposition~24]{AbsilMalick},
\begin{equation*}
L_i
\coloneq R_X^\mathrm{orth}(\proj{\tancone{\R_{\ushort{r}}^{m \times n}}{X}}{X_i-X})
= [U \; U_\perp] \begin{bmatrix} A_i\\ D_i \end{bmatrix} \begin{bmatrix} I_{\ushort{r}} & A_i^{-1} C_i \end{bmatrix} [V \; V_\perp]^\tp,
\end{equation*}
hence
\begin{equation*}
X_i-L_i	= \begin{bmatrix} 0_{\ushort{r} \times \ushort{r}} & 0_{\ushort{r} \times n-\ushort{r}}\\ 0_{m-\ushort{r} \times \ushort{r}} & E_i - D_i A_i^{-1} C_i \end{bmatrix}.
\end{equation*}
Therefore, the second and third terms in~\eqref{eq:AlgebraicDecomposition} respectively equal the first and second terms in~\eqref{eq:GeometricNormalDecomposition}.
An artist view of the decomposition is proposed in Figure~\ref{fig:ThreeTermsDecomposition}. 

\begin{figure}[ht]
\centering
\begin{tikzpicture}[scale=3]
\pgfmathsetmacro\a{1/sqrt(2)}
\def\b{0.7}
\draw [thick] (0, 0) arc (270:210:1);
\draw [thick] (0, 0) arc (270:330:1);
\draw (0, 0) node [below] {$X$} node {{\tiny$\bullet$}};
\draw (\a, 0) node [below right] {$X+\proj{\tancone{\R_{\ushort{r}}^{m \times n}}{X}}{X_i-X}$} node {{\tiny$\bullet$}};
\draw (\a, 0.23) node [right] {$L_i \coloneq R_X^\mathrm{orth}(\proj{\tancone{\R_{\ushort{r}}^{m \times n}}{X}}{X_i-X})$};
\draw (\a, {1-\a}) node {{\tiny$\bullet$}};
\draw (\a, \b) node [right] {$X_i \in \R_{\le r}^{m \times n}$} node {{\tiny$\bullet$}};
\draw (-1, 0.5) node [above right] {$\R_{\ushort{r}}^{m \times n}$};
\draw [thick, -stealth] (0, 0) -- (\a, 0);
\draw [thick, -stealth] (\a, 0) -- (\a, {1-\a});
\draw [thick, -stealth] (\a, {1-\a}) -- (\a, \b);
\draw [thick, -stealth] (0, 0) -- (\a, \b);
\end{tikzpicture}
\caption{Artist view of the decomposition used in the proofs of sections~\ref{subsec:GeometricProof} and~\ref{subsec:AlgebraicProof}. The picture correctly represents the following aspects: $\proj{\tancone{\R_{\ushort{r}}^{m \times n}}{X}}{X_i-X}$ is in $\tancone{\R_{\ushort{r}}^{m \times n}}{X}$, $X_i-L_i$ and $L_i-X-\proj{\tancone{\R_{\ushort{r}}^{m \times n}}{X}}{X_i-X}$ are in $\norcone{\R_{\ushort{r}}^{m \times n}}{X}$, $L_i$ has rank $\ushort{r}$, and $\norm{L_i-X-\proj{\tancone{\R_{\ushort{r}}^{m \times n}}{X}}{X_i-X}}/\norm{\proj{\tancone{\R_{\ushort{r}}^{m \times n}}{X}}{X_i-X}} \xrightarrow{i \to \infty} 0$ since $R^\mathrm{orth}$ is a retraction. However, the picture is unfaithful in several respects: $X_i$, $L_i$, and $X+\proj{\tancone{\R_{\ushort{r}}^{m \times n}}{X}}{X_i-X}$ are, in general, not aligned, and the smooth manifold $\R_{\ushort{r}}^{m \times n}$ does not look like that regardless of $m$, $n$, and $\ushort{r}$.}
\label{fig:ThreeTermsDecomposition}
\end{figure}

\subsection{The tangent cone to the intersection of the real generic determinantal variety and another set: four examples}
\label{subsec:TangentConeIntersection}
The tangent cone to the intersection of two sets is always included in the intersection of the tangent cones~\cite[Theorem~6.42]{RockafellarWets}. In view of Theorem~\ref{thm:TanConeRealGenericDeterminantalVarietyMatrix},~\cite[Theorem~6.1]{CasonAbsilVanDooren2013} shows that the inclusion is an equality for $\R_{\le r}^{m \times n}$ and the unit sphere $\{X \in \R^{m \times n} \mid \norm{X} = 1\}$. Similarly, the tangent cone to the intersection of $\R_{\le r}^{m \times n}$ and an affine subspace of $\R^{m \times n}$ equals the intersection of the tangent cones~\cite[Lemma~3.6]{LiLuo2023}, the tangent cone to the intersection of $\R_{\le r}^{m \times n}$ and a level set of a smooth map $h : \R^{m \times n} \to \R^q$ equals the intersection of the tangent cones if the derivative of $h$ at every point of the level set is surjective and $h(XQ) = h(X)$ for all $X \in \R^{m \times n}$ and $Q \in \st_\R(n, n)$ \cite[Corollary~2]{YangGaoYuan}, and the tangent cone to the intersection of $\R_{\le r}^{n \times n}$ and the convex cone $\{X \in \R^{n \times n} \mid X^\tp = X,\, X \succeq 0\}$ equals the intersection of the tangent cones~\cite[Proposition~3.32]{LevinKileelBoumal2025}.

\bibliographystyle{plainurl}
\bibliography{golikier_bib_abbrv}
\end{document}